\documentclass[oneside,11pt]{amsart}

\title{Generalized Choquet spaces} 

\author{Samuel Coskey}
\address{Samuel Coskey, Department of Mathematics, Boise State University,
1910 University Drive, Boise, ID, 83725-1555, United States}
\email{scoskey@nylogic.org}
\urladdr{boolesrings.org/scoskey}
\author{Philipp Schlicht}
\address{Philipp Schlicht, Mathematisches Institut, Universit\"at Bonn,
Endenicher Allee 60, 53115 Bonn, Germany}
\email{schlicht@math.uni-bonn.de}

\usepackage{mydefs}
\usepackage{etoolbox}
\makeatletter\pretocmd{\@seccntformat}{\S}{}{}\makeatother
\usepackage[marginratio=1:1, hmargin=1.4in]{geometry}
\usepackage{setspace}
\usepackage[textsize=footnotesize,color=blue!40, bordercolor=white]{todonotes}
\usepackage{hyperref}
\DeclareMathOperator{\Iso}{Iso}
\DeclareMathOperator{\Sym}{Sym}
\DeclareMathOperator{\Col}{Col}
\newcommand{\vaan}{V\"a\"an\"anen}

\begin{document}

\begin{abstract}
  We introduce an analog to the notion of Polish space for spaces of weight $\leq\kappa$, where $\kappa$ is an uncountable regular cardinal such that $\kappa^{<\kappa}=\kappa$.  Specifically, we consider spaces in which player~II has a winning strategy in a variant of the strong Choquet game which runs for $\kappa$ many rounds.  After discussing the basic theory of these games and spaces, we prove that there is a surjectively universal such space and that there are exactly $2^\kappa$ many such spaces up to homeomorphism. We also establish a Kuratowski-like theorem that under mild hypotheses, any two such spaces of size $>\kappa$ are isomorphic by a $\kappa$-Borel function. We then consider a dynamic version of the Choquet game and show that in this case the existence of a winning strategy for player~II implies the existence of a winning tactic, that is, a strategy that depends only on the most recent move.  We also study a generalization of Polish ultrametric spaces where the ultrametric is allowed to take values in a set of size $\kappa$. We show that in this context, there is a family of universal Urysohn-type spaces, and we give a characterization of such spaces which are hereditarily $\kappa$-Baire.
\end{abstract}

\maketitle

\section{Introduction}

While descriptive set theory began as the study of definable subsets of the real line, it has since grown into a rich field with applications across many different areas.  One of the keys behind this growth was the observation that the relevant properties of the real line are also available in the broader context of Polish spaces.

Recall that a topological space is said to be \emph{Polish} if it is second countable and completely metrizable.  The Polish space, together with certain relatives such as the Polish metric space and the standard Borel space, turn out to be just the right frameworks for generalizing properties of very well-behaved topological spaces such as the real line $\RR$, the Cantor space ${}^\omega2$, and the Baire space ${}^\omega\omega$.

For example, many classes of countable or separable structures can naturally be parameterized by elements of a Polish space.  It is then possible to use the descriptive set theoretic framework to address questions about the complexity of properties or operations on these structures.  In the subfield known as Borel equivalence relations, one can compare the complexity of classification problems for various classes of structures by studying equivalence relations on spaces of such structures.

Motivated by a desire to generalize applications such as these to larger structures, a great deal of research has been dedicated to finding generalizations and analogs of results from classical descriptive set theory for larger spaces such as ${}^\kappa2$ or ${}^\kappa\kappa$.  Here, each space is endowed with the $\mathord{<}\kappa$-supported product topology. We will always make the assumption that $\kappa^{<\kappa}=\kappa$.

Many results of classical descriptive set theory do not readily generalize to higher cardinals.  Some regularity properties fail for definable subsets of ${}^{\kappa}\kappa$, for example some $\bf\Sigma^1_1$ sets do not have the $\kappa$-Baire property \cite[Theorem 4.2]{halko}.  Moreover for $\kappa>\omega$, continuous images of closed subsets of ${}^{\kappa}\kappa$ have properties unlike those for $\kappa=\omega$, for instance not every image of ${}^{\kappa}\kappa$ under a continuous injection is $\kappa$-Borel \cite{luecke-schlicht-continuous-images}.

Still, many classical results do admit generalizations to higher cardinals, at least, consistently.  For example, although disjoint $\bm{\Sigma}^1_1$ subsets of ${}^\kappa\kappa$ cannot necessarily be separated by a $\bm{\Delta}^1_1$ set, in \cite{vaan} Mekler and \vaan\ were able to establish a weak version of the Suslin separation theorem.  Regarding regularity properties, Schlicht showed in \cite{schlicht-perfect} that it is consistent that all definable subsets of ${}^{\kappa}\kappa$ have the perfect set property.

Many results about Borel equivalence relations have also been generalized.  For instance, a classical result of Lopez-Escobar states that Borel classes of countable structures are axiomatizable by a sentence of the infinitary language $\mathcal L_{\omega_1\omega}$.  In \cite{vaught}, Vaught generalized this result to the case of $\kappa$-Borel classes of structures of size $\kappa$ and the language $\mathcal L_{\kappa^+\kappa}$.  In \cite{tuuri}, Tuuri further generalized Vaught's result to deal with $\dD1$ classes of structures of size $\kappa$.  In \cite{mottoros}, Motto Ros has generalized several results of Louveau-Rosendal concerning analytic quasi-orders.  An overview of the descriptive set theory of ${}^{\kappa}\kappa$ can be found in \cite{vadim}.

Given these successes, it is natural to look for one or more general frameworks, analogous to Polish spaces or standard Borel spaces, in which to carry out generalizations of classical descriptive set theory.  It is natural to define that $X$ is $\kappa$-standard Borel iff there is a $\kappa$-Borel bijection between $X$ and a $\kappa$-Borel subset of ${}^\kappa\kappa$ (this definition has been proposed for instance in \cite{mottoros}).  Moreover in the definition of Polish topological space, the second countability can naturally be replaced with the assumption that the space has weight at most $\kappa$ (\emph{i.e.}, has a basis of size $\leq\kappa$).  However, the completeness of the compatible metric can be replaced by a variety of assumptions.

In this article we consider several such assumptions, with a focus on two of them in particular.  Our first attempt goes via the following classical characterization due to Choquet: $X$ is Polish if and only if $X$ is second countable and player II has a winning strategy in the strong Choquet game.  We will consider the analog of this notion in which the classical Choquet game is replaced with an analogous game of longer length.  Our second attempt will be to define completeness directly by considering only spaces which admit an (generalized) ultrametric which takes values in a set of size $\kappa$.

In a forthcoming work, we will use the framework developed here to study $\kappa$-Choquet groups, $\kappa$-ultrametric groups, and their actions.

The present paper is organized as follows.  In the next section, we introduce a variant of the strong Choquet game which runs for $\kappa$ many rounds.  We then use the game to define a generalization of Polish topological spaces called strong $\kappa$-Choquet spaces, and outline some elementary properties enjoyed by such spaces.  In the third section, we give an analog of Kuratowski's isomorphism theorem which states that assuming $X$ and $Y$ are $\kappa$-Choquet spaces of weight $\leq\kappa$ and size $>\kappa$ such that no point is the intersection of fewer than $\kappa$ many open sets, then $X$ and $Y$ are $\kappa$-Borel isomorphic.  We also prove that there are exactly $2^\kappa$ many $\kappa$-Choquet spaces of weight $\leq\kappa$ up to homeomorphism.

In the fourth section, we consider a dynamic variant of the strong $\kappa$-Choquet game where in each round, instead of playing an open set, the players play a set which is the intersection of fewer than $\kappa$ many open sets. We show that for spaces of weight $\leq\kappa$, the existence of a winning strategy in the dynamic game implies the existence of a winning tactic---a strategy which depends only on the most recent move.

In the fifth section, we study generalizations of Polish ultrametric spaces in which the ultrametric is replaced by a distance function which takes values in a set of size $\kappa$.  We show that, as is the case with Polish ultrametric spaces, there exists a family of universal Urysohn spaces of this type.  Finally, in the last section we give a generalization of a classical result of Debs which characterizes the hereditarily Baire subspaces of the Baire space in terms of their spherically closed subsets.

\section{Generalized Choquet spaces}

In this section, we introduce a generalization of Choquet spaces obtained by lengthening the classical Choquet game.  We then give some of the most basic properties of these spaces.  Unless it is otherwise specified, we always assume the following:
\begin{itemize} 
\item The cardinal $\kappa$ is regular and uncountable, and satisfies $\kappa^{<\kappa}=\kappa$.
\item Topological spaces $X,Y,\ldots$ are Hausdorff and regular.
\end{itemize} 

The Choquet game was originally introduced in \cite{choquet}.  We begin with our generalization.

\begin{defn}
  Let $X$ be a topological space.  The \emph{strong $\kappa$-Choquet game} in $X$ is played by two players, I (sometimes called \textsf{empty}) and II (sometimes called \textsf{nonempty}):
  \begin{center}
    \begin{tabular}{ccccccccccc}
      I  & $U_0,x_0$ & & $U_1,x_1$ & & $\cdots$ & & $U_\lambda,x_\lambda$ & & $\cdots$ & \\
      II & & $V_0$ & & $V_1$ & & $\cdots$ & & $V_\lambda$ & & $\cdots$
    \end{tabular}
  \end{center}
  In the first half of each round, $I$ plays $U_\alpha,x_\alpha$ such that $x_\alpha\in U_\alpha$ and $U_\alpha$ is a relatively open subset of $\bigcap_{\beta<\alpha}U_\beta$.  In the second half of each round, II responds with $V_\alpha$ such that $x_\alpha\in V_\alpha$ and $V_\alpha$ is a relatively open subset of $U_\alpha$. We say that II \emph{wins} the play if for all limit ordinals $\lambda\leq\kappa$, we have $\bigcap_{\alpha<\lambda}U_\alpha\neq\emptyset$.
\end{defn}

Of course, if at any limit ordinal $\lambda<\kappa$ we have $\bigcap_{\alpha<\lambda}U_\alpha=\emptyset$, then $I$ wins immediately; the run cannot and need not continue.

\begin{defn}
  We say that $X$ is a \emph{strong $\kappa$-Choquet} space if player~II has a winning strategy in the strong $\kappa$-Choquet game in $X$.
\end{defn}

We will occasionally also mention \emph{(weak) $\kappa$-Choquet spaces}, where II has a winning strategy in the simpler game in which the points $x_\alpha$ are not played and not used.

The canonical example of a strong $\kappa$-Choquet space is of course the $\kappa$-Baire space ${}^\kappa\kappa$ with the topology generated by the basic open sets of the form
\[N_s=\set{x\in{}^\kappa\kappa\mid s\subset x}
\]
where $s\in{}^{<\kappa}\kappa$.  Of course, the $\kappa$-Cantor space ${}^\kappa2$ with the analogous topology is also a fundamental example.  Each of these spaces is \emph{$\kappa$-additive}, which means that the intersection of fewer than $\kappa$ many open sets is again open.

For an example of a $\kappa$-Choquet space that is not $\kappa$-additive, we often use the linearly ordered spaces $({}^\kappa\kappa,\mathsf{lex})$ and $({}^\kappa2,\mathsf{lex})$ with the topology generated by the lexicographic open intervals.  The next result verifies that these spaces are indeed $\kappa$-Choquet.


\begin{prop}\label{prop:lex}
  The spaces $({}^\kappa\kappa,\mathsf{lex})$ and $({}^{\kappa}2,\mathsf{lex})$ are $1$-dimensional, strong $\kappa$-Choquet spaces.  Moreover, $({}^{\kappa}2,\mathsf{lex})$ is compact.
\end{prop} 

\begin{proof}
  By \cite[Lemma~13.17]{gillman}, every subset of $({}^\kappa2,\mathsf{lex})$ has a least upper bound and greatest lower bound, and this property characterizes compactness for linearly ordered topological spaces.  Similarly, every bounded subset of $({}^\kappa\kappa,\mathsf{lex})$ has a least upper bound and greatest lower bound.  It follows from this property that each of these spaces is not zero-dimensional, and it is clear that each is at most one-dimensional.

  Next, we show that player~II has a winning strategy in the strong $\kappa$-Choquet game in each of these spaces.  Without loss of generality we can suppose that the players play intervals (see Lemma~\ref{strategy versus basic strategy}, below).  Furthermore, it follows from the least upper bound property that the decreasing intersection of closed and bounded intervals is always nonempty.  Thus, in response to a set $U$, player~II can play any open interval $V$ such that $\overline{V}\subset U$, and this strategy will be a winning one.
\end{proof}

We can use this last example to generate $2^\kappa$ many distinct examples of $\kappa$-Choquet spaces.  In Corollary~\ref{number of spaces} below, we will show conversely that there are at most $2^{\kappa}$ many homeomorphism types of such spaces.

\begin{prop}\label{many homeomorphism types} 
  There are at least $2^{\kappa}$ many homeomorphism types of connected strong $\kappa$-Choquet spaces of weight $\leq\kappa$.
\end{prop} 

\begin{proof}
  For each $A\subset \kappa$, we will construct a space $X(A)$ in such a way that no two of them are homeomorphic.  In our construction, we let $X=({}^{\kappa}\kappa,\mathsf{lex})$ and let $\overline{X}=X$ together with a maximum element.  To build $X(A)$ begin with a copy of $X$, and for each $\alpha\in A$ attach another copy of $X$ to $\alpha^\smallfrown0^{\kappa}$, and for each $\alpha\notin A$ attach a copy of $\overline{X}$ to $\alpha^\smallfrown0^{\kappa}$.  Then $X(A)$ is strong $\kappa$-Choquet, since player~II can follow the winning strategy for $X$ as long as player~I plays points in the base copy, and follow the winning strategy for one of the new copies of $X$ once a point played by player~I leaves the base copy.

  To see that no two of these are homeomorphic, note that every non-endpoint of $X$ or $\overline{X}$ is a \emph{cut point}, \emph{i.e.},\ removing it from the space leaves exactly two connected components. Since any homeomorphism preserves cut points, $X(A)$ and $X(B)$ are not homeomorphic if $A\neq B$.
\end{proof} 

As we saw in the proof of Proposition~\ref{prop:lex}, we will occasionally have use for the following result, which states that in the $\kappa$-Choquet game we may assume that the players use basic open sets instead of just open sets.

\begin{lem}\label{strategy versus basic strategy} 
  Suppose that $X$ has weight $\leq\kappa$.  Suppose that one of the players has a winning strategy in the (strong) $\kappa$-Choquet game. Then this player has a winning strategy in which she only plays basic open sets (as usual, intersected with the run so far).
\end{lem}

\begin{proof}
  We only consider the case of player~II in the strong $\kappa$-Choquet game.  Let $\sigma$ be a winning strategy for player~II in the strong $\kappa$-Choquet game.  Consider the modified strategy in which player~II always plays a basic open subset of the set that she would have played according to $\sigma$. It is clear that if $U_0,x_0,V_0,U_1,x_1,V_1\ldots$ is a run of this game, and $\lambda$ is a limit ordinal, then
\[\bigcap_{\alpha<\lambda}V_\alpha=\bigcap_{\alpha<\lambda}U_\alpha=\bigcap_{\alpha<\lambda}\sigma(U_0,x_0,\ldots,U_\alpha,x_\alpha)
\]
This implies that the modified strategy consists of valid plays, and that it is a winning strategy if $\sigma$ is.
\end{proof} 

Recall that a space is \emph{$\kappa$-Baire} if the intersection of $\kappa$ many dense open sets is dense, and \emph{weakly $\kappa$-Baire} if the intersection of $\kappa$ many dense open sets is nonempty.  Although each of the example spaces listed above has the additional property that it is $\kappa$-Baire, it worth noting that not all $\kappa$-Choquet spaces turn out to be weakly $\kappa$-Baire.  For instance, ${}^\omega\omega$ is a strong $\kappa$-Choquet space for somewhat trivial reasons, but assuming as usual that $\kappa^{<\kappa}=\kappa$, then ${}^\omega\omega$ is not $\kappa$-Baire for $\kappa>\omega$.

\begin{prop}[cf.\ {\protect\cite[Theorem~8.11]{kechris}}] 
  \label{characterization of Baire by games}
  Suppose that $X$ is a space of weight $\leq\kappa$ such that the intersection of any decreasing sequence of fewer than $\kappa$ many open sets has nonempty interior. Then we have:
\begin{enumerate}
\item $X$ is not $\kappa$-Baire if and only if player I has a winning strategy in the modified version of the weak $\kappa$-Choquet game where player II plays first at limits but the winning condition remains the same.
\item $X$ is not weakly $\kappa$-Baire if and only if player I has a winning strategy in the modified version if the weak $\kappa$-Choquet game where player II begins and plays first at limits but the winning condition remains the same.
\end{enumerate} 
\end{prop}


\begin{proof} 
  We address only part (a).  Suppose first that $X$ is not $\kappa$-Baire. Then there exists a sequence $(U_{\alpha})_{\alpha<\kappa}$ of dense open sets such that $\bigcap_{\alpha<\kappa} U_{\alpha}$ is not dense, so there is a nonempty basic open set $U$ with $U\cap \bigcap_{\alpha<\kappa}U_{\alpha}=\emptyset$. Let player I play $U\cap U_0$ in her first move. Then by hypothesis on $X$, it is valid for player~I to simply play $U_{\alpha}$ (intersected with the run so far) in each round $\alpha$. Moreover, this is clearly a winning strategy for player~I.

  Conversely, suppose that player~I has a winning strategy $\sigma$ for this game.  We will construct a $\kappa$-closed subtree $T\subset{}^\kappa\kappa$ and sets $U_s$ for $s\in T$ successor such that:
\begin{enumerate}
\item For each $b\in[T]$ the sequence $U_{b\restriction0},U_{b\restriction1},\ldots$ forms a valid sequence of moves for player~I according to $\sigma$; 
\item $U_s\cap U_t=\emptyset$ for $s\neq t$ in $T$ of the same length; and
\item For each $\alpha$ successor, the set $G_{\alpha+1}=\bigcup\set{U_s\mid\dom(s)=\alpha+1}$ is dense and open in $G_{\alpha}=\bigcup\set{U_s\mid\dom(s)=\alpha}$.
\end{enumerate}
To carry out the construction, we initially let $U_\emptyset$ denote the starting move of player~I.  Given $s\in T$ such that $\dom(s)$ is a successor, we let $U_{s^\smallfrown\beta}$ enumerate a maximal pairwise disjoint sequence of $\sigma$'s responses to a valid player~II response to the run $U_{s\restriction0},\ldots,U_{s\restriction\alpha},\ldots,U_s$ ($\alpha$ a successor ordinal).  We close $T$ under limits, and for $s\in T$ such that $\dom(s)$ is limit we define its successors $U_{s^\smallfrown\beta}$ similarly.
    
  To show that $X$ is not $\kappa$-Baire, let $\gamma\leq\kappa$ be least such that $\bigcap_{\alpha<\gamma} G_{\alpha}$ is not dense in $U_{\emptyset}$.  If $\gamma<\kappa$, then $\bigcap_{\alpha<\gamma} G_{\alpha}$ is not dense and hence $X$ is not $\kappa$-Baire.  If $\gamma=\kappa$, then $\bigcap_{\alpha<\kappa} G_{\alpha}$ is empty, since $\sigma$ is a winning strategy for player I, and hence $X$ is not $\kappa$-Baire.
\end{proof}

We remark that as in the classical case, in Proposition~\ref{characterization of Baire by games}(a) the weak $\kappa$-Choquet game cannot be replaced by the strong $\kappa$-Choquet game.  To see this, let $T$ be a subtree of ${}^{<\kappa}\kappa$ which is isomorphic to ${}^{<\kappa}2$ and let $D$ be a dense subset of $[T]$ of size $\kappa$.  We consider the space $X=({}^\kappa\kappa\smallsetminus[T])\cup D$.  Then $X$ is $\kappa$-Baire and so player~I does not have a winning strategy in the weak $\kappa$-Choquet game on $X$.  But there is a strategy for player~I in the strong $\kappa$-Choquet game by playing basic open subsets of $[T]$ while successively avoiding elements of $D$.

We now briefly address inheritance and preservation of the strong $\kappa$-Choquet property. The strong $\kappa$-Choquet property is clearly inherited by open subsets, as well as subsets which are the intersection of $<\kappa$ many open sets. However, the property is not necessarily inherited by closed subsets. For example, the subspace of ${}^\kappa\kappa$ consisting of those $x$ such that $x(i)\neq0$ for all but finitely many $i<\kappa$ is neither $\omega$-Choquet nor $\omega$-Baire.  The following proposition gives some preservation properties that do hold for strong $\kappa$-Choquet spaces.

\begin{prop}[cf.\ \protect{\cite[Theorem~4.1.2]{gao}}]
  \label{images of strong Choquet spaces}
  \begin{enumerate}
  \item If $X$ is strong $\kappa$-Choquet and $f\colon X\to Y$ is continuous, open, and surjective, then $Y$ is strong $\kappa$-Choquet.
  \item The $\mathord<\kappa$-supported product of $\kappa$ many strong $\kappa$-Choquet spaces is strong $\kappa$-Choquet.
  \end{enumerate}
\end{prop} 

\begin{proof}[Proof outline] 
  (a) Given a winning strategy $\tau$ for II in $X$, we construct a winning strategy for II in $Y$ as follows.  Given a run $U_0,y_0,\ldots,U_\alpha,y_\alpha$, let $V_\alpha$ be the result of $\tau$ applied to the run $f^{-1}(U_0),x_0,\ldots,f^{-1}(U_\alpha),x_\alpha$, where $x_\beta$ is any point in $f^{-1}(y_\beta)$.  We then let II respond with $f(V_\alpha)$.  Since $f$ is open, this is a valid move, and it is easy to see that $\tau$ is winning implies this strategy is winning.
    
  (b) Given a run $U_0,x_0,\ldots,U_\alpha,x_\alpha$ in $\prod_{i\in I}X_i$ with $\alpha<\kappa$, we can suppose without loss of generality that each $U_\beta$ for $\beta\leq\alpha$ is equal to a basic open set of the form $\prod_{i\in I} U_{\beta,i}$, where there is $I_\beta\subset I$ with $|I_\beta|<\kappa$ such that  $U_{\beta,i}=X_i$ for all $i\notin I_{\beta}$, intersected with the run so far.  Let $\pi_i\colon \prod_{j\in I}X_j\rightarrow X_i$ denote the projection onto the $i^{\mathrm{th}}$ coordinate.  For each $i$, the sequence $U_{0,i},\pi_i(x_0),\ldots,U_{\alpha,i},\pi_i(x_\alpha)$ determines a run of the strong Choquet game in $X_i$, for which player~II has a response $V_{\alpha,i}$ given by a winning strategy. We can suppose that $V_{\alpha,i}=X_i$ for all $i\notin I_{\alpha}$, so that in $X$ we may let player~II respond with $V_\alpha=\prod_{i\in I} V_{\alpha,i}$. Since player~II wins the game in each $X_i$, the set $\bigcap_{\alpha<\gamma} V_{\alpha,i}$ is nonempty for every $i\in I$ and $\gamma\leq\kappa$, and it follows that $\bigcap_{\alpha<\gamma} V_{\alpha}=\prod_{i\in I}\bigcap_{\alpha<\gamma} V_{\alpha,i}$ is nonempty.
\end{proof} 


\section{Borel isomorphisms} 

In this section we will establish a generalization of the Kuratowski isomorphism theorem for strong $\kappa$-Choquet spaces.  Recall that the classical Kuratowski theorem states that all uncountable Polish spaces are Borel isomorphic.  We begin with a (weak) version of the Cantor embedding theorem, which gives conditions under which ${}^\kappa2$ embeds into a given strong $\kappa$-Choquet space.

Our proof of the Cantor embedding theorem will be the first to illustrate the method by which we eliminate the use of a metric from classical arguments.  When a sequence of shrinking balls is used, we use our assumption that the space has weight $\leq\kappa$.  When completeness is used to find a point in the intersection of a family of closed sets, we use the Choquet property instead.  For a comparison with the classical argument, see for instance Theorem~1.3.6 of \cite{gao}.

Our result differs from the classical one in that we need to assume the given space $X$ is $\kappa$-perfect.  Here, we say that $x\in X$ is \emph{$\kappa$-isolated} if $\{x\}$ can be written as the intersection of fewer than $\kappa$ many open sets, and that $X$ is \emph{$\kappa$-perfect} if $X$ has no $\kappa$-isolated points.  It is easy to see that if $X$ has weight $\kappa$ (and as always $\kappa^{<\kappa}=\kappa$) then $X$ contains at most $\kappa$ many $\kappa$-isolated points.  (Indeed, there are only $\kappa$ many intersections of basic open sets of length $<\kappa$.)  The hypothesis that $X$ is $\kappa$-perfect in Proposition~\ref{prop:injection} is needed because the Choquet property is not necessarily inherited by the set of $\kappa$-nonisolated points. (For instance, let $T_0$ denote the subtree of ${}^{<\kappa}2$ consisting of just those sequences with finitely many $0$'s, and let $T=T_0$ together with a branch of length $\kappa$ added on top of every $\omega$-path in $T_0$.  Then $[T]$ is strong $\kappa$-Choquet, but its perfect kernel is $[T_0]$ is not.)

\begin{prop}[cf.\ {\protect\cite[Theorem~1.3.6]{gao}}]
  \label{prop:injection}
  If $X$ is a nonempty $\kappa$-perfect $\kappa$-Choquet space with weight $\leq\kappa$, then there is a continuous injection from $2^\kappa$ into $X$.
\end{prop}

\begin{proof}
  Let $\sigma$ be a winning strategy for player~II in the Choquet game in $X$.  Let $(B_\alpha,C_\alpha)$ enumerate the set of pairs of basic open sets such that $\overline{B_\alpha}\subset C_\alpha$.  We construct families of subsets $U_s$ and $V_s$ for $s\in {}^{\mathord{<}\kappa}2$ (with $U_\emptyset=X$) satisfying:
\begin{enumerate}
\item $U_{s^\smallfrown0}$ and $U_{s^\smallfrown1}$ are relatively open subsets of $V_s$ such that $\overline{U_{s^\smallfrown0}}$ and $\overline{U_{s^\smallfrown1}}$ are disjoint;
\item\label{emb:welldef} if $\dom(s)=\alpha$ then $U_s$ is either contained in $C_\alpha$ or disjoint from $\overline{B_\alpha}$;
\item for $s$ of limit length $\lambda$, $U_s=\bigcap_{\alpha<\lambda}U_{s\restriction\alpha}$; and
\item\label{emb:choquet} for each $s$, $V_s=\sigma(U_{s\restriction0},U_{s\restriction1},\cdots,U_s)$.
\end{enumerate}
The construction is possible because the space is $\kappa$-perfect.
  
Now, for $x\in{}^\kappa2$, the set $\bigcap_{\alpha<\kappa}U_{x\restriction\alpha}$ is nonempty by \eqref{emb:choquet} and it is a singleton by \eqref{emb:welldef}.  Hence, we may let $f(x)=$ this unique element, and it is clear that $f$ is one-to-one.  To see that it is continuous, suppose that $f(x)\in O$ and $O$ is open.  Choose a pair $(B_\alpha,C_\alpha)$ such that $f(x)\in B_\alpha$ and $C_\alpha\subset O$.  Then since $f(x)\in U_{x\restriction\alpha+1}$ we must have had $U_{x\restriction\alpha+1}\subset C_\alpha$.  Hence for all $y\in{}^\kappa2$ with $y\restriction\alpha+1=x\restriction\alpha+1$ we have $f(y)\in C_\alpha\subset O$.
\end{proof}

Since there is a continuous injection from $\kappa^\kappa$ into $2^{\kappa}$, we can even get a continuous injection from $\kappa^\kappa$ into $X$.

\begin{question}
  If $X$ is a $\kappa$-perfect $\kappa$-Choquet space, then is there a continuous injection from $2^\kappa$ into $X$ with closed image?
\end{question} 

The next two results use similar arguments to represent a given strong $\kappa$-Choquet space as a surjective image.

\begin{thm}[cf.\ \protect{\cite[Theorem~1.3.7]{gao}}]
  \label{thm:bij}
  If $X$ is strong $\kappa$-Choquet with weight $\leq\kappa$, then there is a subtree $T\subset{}^{<\kappa}\kappa$ without end nodes and a continuous bijection $f\colon [T]\rightarrow X$.
\end{thm}

\begin{proof}
  As before, let $\sigma$ be a winning strategy for player~II in the strong $\kappa$-Choquet game in $X$, and let $(B_{\alpha},C_{\alpha})_{\alpha<\kappa}$ enumerate the pairs of basic open sets in $X$ such that $\overline{B_{\alpha}}\subset C_{\alpha}$.  This time we construct a subtree $T\subset {}^{<\kappa}\kappa$ with no end nodes, subsets $U_s,U_s'\subset X$ for $s\in T$, and elements $x_s\in X$ for $s\in T$ and $\dom(s)$ a successor, such that the following properties hold:
\begin{enumerate} 
\item\label{bij:choquet} for each $b\in[T]$ the sequence $(U_{b\restriction0},x_{b\restriction1},U_{b\restriction1},x_{b\restriction2}\ldots)$ is a valid sequence of plays for player~I in a run of the strong $\kappa$-Choquet game in which player~II plays by $\sigma$;
\item\label{bij:welldef} if $\dom(s)=\alpha$ then $U_s$ is either contained in $C_\alpha$ or disjoint from $\overline{B_\alpha}$;
\item\label{bij:refinement} for $s\subset t\in T$ we have $U_t'\subset U_s'\subset U_s$; and
\item\label{bij:partition} for each $\alpha$, $\set{U_s'\mid s\in   T,\dom(s)=\alpha}$ is a partition of $X$.
\end{enumerate} 

  To carry out the construction, we begin by letting $U_\emptyset=U'_\emptyset=X$. Next, if $s\in T$ with $\dom(s)=\alpha$ and $U_s,U_s',x_s$ have been defined, we define the immediate successors $U_{s^\smallfrown\beta}$ as follows. For $x\in U_s'$, first let
\[W_x=\begin{cases}\sigma(U_{s\restriction0},x_{s\restriction1},U_{s\restriction1},x_{s\restriction2},\ldots,U_s,x)\ \cap C_\alpha &\mathrm{if\ }x\in\overline{B_\alpha}; \\
\sigma(U_{s\restriction0},x_{s\restriction1},U_{s\restriction1},x_{s\restriction2},\ldots,U_s,x) \smallsetminus\overline{B_\alpha}&\text{otherwise.}\end{cases}
\]
Note that the family of all $W_x$ for $x\in U_s'$ covers $U_s'$. We let $U_{s^\smallfrown\beta}$ enumerate a minimal subcover and $x_{s^\smallfrown\beta}$ the corresponding $x$'s.  We also define
\[U_{s^\smallfrown\beta}'=U_{s^\smallfrown\beta}\smallsetminus\bigcup\set{U_{s^\smallfrown\gamma}\mid\gamma<\beta}\;.
\]

  If $s$ has limit length $\lambda$ and $s\restriction\alpha\in T$ for all $\alpha<\lambda$, then we define $U_s=\bigcap_{\alpha<\lambda}U_{s\restriction\alpha}$ and also $U_s'=U_s\smallsetminus\bigcup\set{U_t\mid t<_{\mathsf{lex}}s}$ (as in the successor case).  Finally, we put $s\in T$ if and only if $U_s'\neq\emptyset$.  It is straightforward to verify that if $s\in T$ has limit length $\lambda$, then $U_s'=\bigcap_{t\subsetneq s}U_t'$, and furthermore that the requirements \eqref{bij:refinement} and \eqref{bij:partition} are fulfilled. This completes the construction.

  Now, for all branches $b\in[T]$ the set $\bigcap_{\alpha<\kappa}U_{b\restriction\alpha}$ is nonempty by \eqref{bij:choquet} and it is a singleton by \eqref{bij:welldef}. Hence, we may let $f(b)$ denote this unique element.  To see that $f$ is continuous, suppose that $f(b)\in O$ and $O$ is open. Choose a pair $(B_\alpha,C_\alpha)$ such that $f(b)\in B_\alpha$ and $C_\alpha\subset O$.  Then since $f(b)\in U_{b\restriction\alpha+1}$, we must have had $U_{b\restriction\alpha+1}\subset C_\alpha$.  Hence for all $c\in[T]$ with $c\restriction\alpha+1=b\restriction\alpha+1$ we have $f(c)\in C_\alpha\subset O$.
  
  To show that $f$ is injective, we first claim that $\bigcap_{\alpha<\kappa}
  U_{b\upharpoonright\alpha}=\bigcap_{\alpha<\kappa}U_{b\upharpoonright\alpha}'$ for every branch $b\in[T]$. Otherwise $f(b)\in U_{b\upharpoonright\alpha}\smallsetminus U_{b\upharpoonright\alpha}'$ for some $\alpha<\kappa$ and we can find some $s<_{\mathsf{lex}}b\upharpoonright\alpha$ with $f(b)\in U_s$. Since $x_{b\upharpoonright\beta}\in U_{b\upharpoonright\beta}'$ and thus $x_{b\upharpoonright\beta}\notin U_s$ for all $\beta\geq\alpha$, and since $U_s$ is open, we have $f(b)=\lim_{\alpha<\kappa}x_{b\upharpoonright\alpha}\notin U_s$. Thus $f$ is injective by the requirement \eqref{surj:partition}.

  To show that $f\colon[T]\rightarrow X$ is surjective, suppose that $x\in X$ and find the unique $s_{\alpha}\in {}^{\alpha}2$ with $x\in U_{s_\alpha}'$ for each $\alpha<\kappa$. Then $s_{\alpha}\subset s_{\beta}$ for all $\alpha<\beta<\kappa$. Hence $f(b)=x$ for the unique branch $b$ through all $s_{\alpha}$.
\end{proof} 

\begin{cor}\label{standard Borel}
  If $X$ is strong $\kappa$-Choquet and of weight $\leq\kappa$, then $X$ is standard Borel---that is, $X$ is in $\kappa$-Borel bijection with a $\kappa$-Borel subset of ${}^\kappa\kappa$.
\end{cor}


We next show that the $\kappa$-Baire space is surjectively universal among strong $\kappa$-Choquet spaces.

\begin{thm}
  \label{thm:surj}
  If $X$ is nonempty and strong $\kappa$-Choquet with weight $\leq\kappa$, then there is a continuous surjection $f\colon{}^{\kappa}\kappa \rightarrow X$.  If $X$ is additionally $\kappa$-additive, then $f$ can be chosen open as well.
\end{thm} 

\begin{proof}
  The construction of a continuous surjection $f\colon{}^{\kappa}\kappa\rightarrow X$ is similar to the previous proof and we will use the same notation. This time we construct subsets $U_s\subset X$ for $s\in {}^{\kappa}\kappa$, and elements $x_s\in X$ for $s\in {}^{\kappa}\kappa$ successor satisfying:
\begin{enumerate} 
\item\label{surj:choquet} for each $b\in{}^{\kappa}\kappa$ the sequence $(U_{b\restriction0},x_{b\restriction1},U_{b\restriction1},x_{b\restriction2}\ldots)$ is a valid sequence of plays for I in a run of the strong $\kappa$-Choquet game in which II plays by $\sigma$;
\item\label{surj:welldef} if $\dom(s)=\alpha$ then $U_s$ is either contained in $C_\alpha$ or disjoint from $\overline{B_\alpha}$;
\item\label{surj:refinement} for $s\subset t\in {}^{<\kappa}\kappa$ we have $U_t\subset U_s$; and
\item\label{surj:partition} for each $s\in {}^{<\kappa}\kappa$, $\set{U_{s^\smallfrown\beta}\mid\beta<\kappa}$ covers $U_s$.
\end{enumerate} 
  

If $s\in T$ and $U_s,x_s$ have been defined, we define the immediate successors $U_{s^\smallfrown\alpha}$ as follows. For $x\in U_s$, define $W_x$ as in the previous proof, so that the set of $W_x$ for $x\in U_s$ covers $U_s$.  This time we let $U_{s^\smallfrown\beta}$ enumerate a subcover of size $\kappa$, with repetitions allowed, and $x_{s^\smallfrown\beta}$ be the corresponding $x$'s.  For $s$ of limit length $\lambda$ we define $U_s=\bigcap_{\alpha<\lambda}U_{s\restriction\alpha}$.

  We may now define $f$ as in the previous proof, and it will be continuous and surjective by the same arguments.
\end{proof} 

The above results have a number of consequences. The first provides the upper bound of $2^\kappa$ on the number of strong $\kappa$-Choquet spaces up to homeomorphism.  The second is the promised Kuratowski-like result which gives conditions under which any two strong $\kappa$-Choquet spaces of size $>\kappa$ are isomorphic by a $\kappa$-Borel function.

\begin{cor}\label{number of spaces}
  There are exactly $2^{\kappa}$ many homeomorphism types of strong $\kappa$-Choquet space of weight $\leq\kappa$.
\end{cor} 

\begin{proof}
  We have already shown in Proposition \ref{many homeomorphism types} that there are at least $2^\kappa$ many such spaces.  To see that there are at most $2^\kappa$, by Theorem~\ref{thm:surj} there is a continuous surjection $f\colon{}^{\kappa}\kappa\rightarrow X$.  Let $(B_{\alpha})_{\alpha<\kappa}$ denote a base for $X$ and let $\mathcal U=\set{f^{-1}(B_\alpha)\mid\alpha<\kappa}$. Then the topology of $X$ is determined up to homeomorphism by $f$ and $\mathcal U$.

  Now, since a continuous function is determined by its restriction to a dense subset, there are at most $2^{\kappa}$ many continuous maps with domain ${}^{\kappa}\kappa$.  Similarly, there are at most $2^{\kappa}$ many possible open sets $\mathcal U$.  Thus there at most $2^{\kappa}$ many possible pairs $(f,\mathcal U)$ and so at most $2^\kappa$ many possible topologies on $X$.
\end{proof} 

\begin{cor}\label{cor:iso}
  \begin{enumerate}
  \item Any two $\kappa$-perfect, strong $\kappa$-Choquet spaces of weight $\leq\kappa$ and of size $>\kappa$ are $\kappa$-Borel isomorphic.
  \item Suppose that any 
subtree $T\subset{}^{<\kappa}\kappa$ with $\abs{[T]}>\kappa$ has a perfect binary subtree.  Then any two strong $\kappa$-Choquet spaces of weight $\leq\kappa$ and size $>\kappa$ are $\kappa$-Borel isomorphic.
  \end{enumerate}
\end{cor} 

\begin{proof}
  Given any strong $\kappa$-Choquet space $X$ of size $>\kappa$, we first claim that the map $[T]\to X$ constructed in the proof of Theorem~\ref{thm:bij} is a $\kappa$-Borel isomorphism. We have already said it is continuous and bijective, but it is also easy to see that it maps basic open sets to $\kappa$-Borel sets.  Indeed, as we showed in the proof of Theorem~\ref{thm:bij}, $\bigcap_{\alpha<\kappa}U_{b\upharpoonright\alpha}=\bigcap_{\alpha<\kappa}U_{b\upharpoonright\alpha}'$ for every branch $b\in[T]$, and thus the image of the basic open set given by $t\in T$ is the $\kappa$-Borel set $U_t'=U_t\smallsetminus(\bigcup_{s<_{\mathsf{lex}}t}U_s)$.

  In the case that $X$ is $\kappa$-perfect, then it is additionally straightforward to check that for the injection $f\colon{}^{\kappa}2\rightarrow X$ constructed in Proposition~\ref{prop:injection}, the image of any basic open set is $\kappa$-Borel.  Moreover, it is easy to see that ${}^{\kappa}\kappa$ is homeomorphic to a $G_{\delta}$ subset of ${}^{\kappa}2$. (Here, $G_{\delta}$ means an intersection of $\kappa$ many open sets.) Thus we can apply a Cantor--Schr\"oder--Bernstein argument to conclude that $X$ is $\kappa$-Borel isomorphic to $\kappa^\kappa$, as desired.
  
  If instead we assume the hypothesis in part (b), then ${}^{\kappa}2$ is homeomorphic to a closed subset of $[T]\subset{}^{\kappa}\kappa$.  Hence $[T]$, and therefore $X$, is $\kappa$-Borel isomorphic to ${}^{\kappa}\kappa$ by a Cantor--Schr\"oder--Bernstein argument.
\end{proof} 

We cannot drop the set-theoretic hypothesis in Corollary~\ref{cor:iso}(b).  Indeed, suppose that $T_0$ is a $\kappa$-Kurepa tree with $\kappa^+$ many branches and $\kappa^+<2^\kappa$, and let $T=T_0$ together with a branch added on top of every path through $T_0$ of limit length $<\kappa$. Then $[T]$ is a strong $\kappa$-Choquet space, but since it has cardinality $\kappa^+<2^\kappa$, it is not isomorphic to ${}^\kappa2$.  We remark that by a result of Silver, the hypothesis in Corollary~\ref{cor:iso}(b) does hold after collapsing an inaccessible $\lambda$ to $\kappa^+$ using the L\'evy collapse.

\begin{cor}
  Suppose that $G$ is $\Col(\kappa,\mathord{<}\lambda)$-generic over $V$, where $\lambda>\kappa$ is inaccessible in $V$, and work in $V[G]$.  Suppose that $X$ is a strong $\kappa$-Choquet space of weight $\leq\kappa$.  If $A\subset X$ is definable from ordinals and subsets of $\kappa$ and $|A|>\kappa$, then there is a continuous injection $f\colon {}^{\kappa}2\rightarrow A$.
\end{cor} 

\begin{proof}
  Again, we have a continuous bijection $f\colon[T]\to X$ from Theorem~\ref{thm:bij} where $T\subset{}^{<\kappa}\kappa$.  By \cite{schlicht-perfect}, in this model the perfect set property holds for all subsets of ${}^\kappa\kappa$ that are definable from ordinals and subsets of $\kappa$, in particular for $f^{-1}(A)$. Hence there exists a perfect binary subtree $S\subset T$ such that $[S]\subset f^{-1}(A)$.
\end{proof}

We close this section with the question of whether there is a single strong $\kappa$-Choquet space into which all of them embed as a closed subspace.  Ilmavirta has shown in \cite{Ilmavirta} that ${}^\kappa2$ is universal for $\kappa$-additive spaces of weight $\leq\kappa$.  However, this space is not universal for all strong $\kappa$-Choquet spaces, and we have yet to find one that is. By theorem \ref{thm:surj}, $\kappa^\kappa$ is surjectively universal for strong $\kappa$-Choquet spaces of weight $\leq\kappa$.

\begin{question}
  Is there a universal strong $\kappa$-Choquet space of weight $\leq\kappa$?
\end{question}

In Section~5 we will discuss the analog of this question for generalized ultrametric spaces.

\section{Dynamic games} 

In this section, we consider a generalization of the strong $\kappa$-Choquet game where instead of playing open sets, each player may play the intersection of $<\lambda$ many open sets in each round, where $\lambda$ is a fixed cardinal $\leq\kappa$. We discuss the implications between the corresponding types of spaces as $\lambda$ varies.  We show that when $\lambda=\kappa$, the existence of a winning strategy in this game is equivalent to the existence of a winning tactic.

\begin{defn}
  A \emph{tactic} for player~II in the strong $\kappa$-Choquet game is a strategy which depends only on the most recent move of player~I.
\end{defn} 

In the classical strong Choquet game on a separable space, the existence of a winning strategy for player II implies the existence of a complete metric, and therefore the existence of a winning tactic. We will see that this is also true for the $\kappa$-dynamic version of the strong $\kappa$-Choquet game, which we define presently.

\begin{defn} 
  \begin{enumerate} 
  \item Suppose that $\lambda\leq\kappa$.  The \emph{$\lambda$-dynamic} strong $\kappa$-Choquet game is played as the strong $\kappa$-Choquet game, except that rather than open sets, the players may play the intersection of fewer than $\lambda$ many open sets (as usual intersected with the run up until that point).
  \item A space $X$ is said to be \emph{$\lambda$-dynamic strong $\kappa$-Choquet} if player~II has a winning strategy in the $\lambda$-dynamic strong $\kappa$-Choquet game on $X$.
  \end{enumerate}
\end{defn} 

Of course, the $\omega$-dynamic game is simply the ordinary game. Importantly, all of the results in the previous two sections hold just as well with strong $\kappa$-Choquet spaces replaced by $\lambda$-dynamic strong $\kappa$-Choquet spaces.  For instance, we have the following:
\begin{itemize}
\item The $\lambda$-dynamic strong $\kappa$-Choquet property is preserved by continuous open images (see Proposition~\ref{images of strong Choquet spaces}).
\item If $X$ is $\lambda$-dynamic strong $\kappa$-Choquet and of weight $\leq\kappa$ then $X$ is standard Borel (see Corollary~\ref{standard Borel}).
\item If $X$ is additionally $\kappa$-perfect and of size $>\kappa$ then $X$ is $\kappa$-Borel isomorphic with ${}^\kappa\kappa$ (see Corollary~\ref{cor:iso}).
\end{itemize}

We continue with further properties of the $\lambda$-dynamic games.

\begin{prop}\label{implications for lambda dynamic}
  Suppose that $\lambda, \mu$ are cardinals with $\omega\leq\lambda<\mu\leq\kappa$.
  \begin{enumerate} 
  \item Every $\lambda$-dynamic strong $\kappa$-Choquet space is $\mu$-dynamic strong $\kappa$-Choquet.
  \item If $\lambda$ is regular, then there is a $\mu$-dynamic strong $\kappa$-Choquet space of weight $2^{\mathord{<}\lambda}$ which is not $\lambda$-dynamic strong $\kappa$-Choquet.
  \end{enumerate} 
\end{prop} 

\begin{proof} 
  (a) Given a winning strategy for player~II in the $\lambda$-dynamic strong $\kappa$-Choquet game, we can easily derive a winning strategy for player~II in the $\mu$-dynamic strong $\kappa$-Choquet game as follows. Suppose that at some stage player~I plays $(W,x)$ where $W=\bigcap_{\alpha<\theta}W_\alpha$, each $W_\alpha$ is open, and $\theta<\mu$. Then player~II turns to a side-run of the $\lambda$-dynamic game where she instructs player~I to play $(W_\alpha,x)$ sequentially and responds with $V_\alpha$ according to his winning strategy there. She then responds in the $\mu$-dynamic game with $\bigcap_{\alpha<\theta}V_\alpha$. Since $\theta<\mu$, this is a valid move for player~II and it is easy to see this strategy is winning too.

  (b) Let $\mathcal{C}_{\lambda}$ denote the set of all $x\in{}^{\lambda}2$ such that $\{\alpha<\lambda\mid x(\alpha)=1\}$ contains a club in $\lambda$. 
  Then Player~I has a winning strategy in the $\lambda$-dynamic game by simply extending with a $0$ in limit steps $<\lambda$ (as in \cite[Theorem~4.2]{halko}). On the other hand, player~II has a winning strategy in the $\mu$-dynamic game since player~II can always answer the first move with a singleton. 
\end{proof} 
 
Given a topological space $X$, we define the \emph{$\lambda$-topology} on $X$ to be the topology generated by sets which are the intersection of fewer than $\lambda$ many open sets.  If $X$ has weight $\leq\kappa$ and $\lambda\leq\kappa$, then the $\lambda$-topology on $X$ has weight $\leq\kappa$ as well (as usual, $\kappa^{<\kappa}=\kappa$).

\begin{lem}
  Suppose that $\lambda\leq\kappa$ and that $X$ is a space of weight
  $\leq\kappa$. Then $X$ is $\lambda$-dynamic strong $\kappa$-Choquet
  if and only if its $\lambda$-topology is strong $\kappa$-Choquet.
\end{lem} 


We now come to the promised result that in the $\kappa$-dynamic game,
winning strategies give rise to winning tactics.

\begin{thm}
  \label{strategy to tactic}
  Suppose that $X$ is $\kappa$-dynamic strong $\kappa$-Choquet of weight $\leq\kappa$. Then there is a winning tactic for player~II in the $\kappa$-dynamic strong $\kappa$-Choquet game on $X$.
\end{thm}

\begin{proof}
  Since the $\kappa$-topology on $X$ is strong $\kappa$-Choquet, there
  exists a continuous surjection $f\colon {}^{\kappa}\kappa\rightarrow
  X$ by Proposition~\ref{thm:bij}.  Since the $\kappa$-topology on $X$
  is $\kappa$-additive, we can assume that $f$ maps open sets to
  $\kappa$-open sets.  Since ${}^{\kappa}\kappa$ is $\kappa$-additive
  as well, every $\kappa$-open set is mapped to a $\kappa$-open set.
  Note that player II has a winning tactic in the strong
  $\kappa$-Choquet game on ${}^{\kappa}\kappa$, in fact player~II wins
  every run no matter what she plays.
  As in the proof of Proposition~\ref{images of strong Choquet
    spaces}, this tactic can be transferred to a winning tactic for
  player~II in the $\kappa$-dynamic strong $\kappa$-Choquet game on
  $X$.
\end{proof} 


\section{Generalized ultrametric spaces}

In this section we discuss a generalization of Polish ultrametric spaces to higher cardinalities; similar spaces have been studied in \cite{delon}, and extensively in a series of articles beginning with \cite{priess}.  We generalize several properties of ultrametric spaces and their isometry groups to $\kappa$-ultrametric spaces, which we presently define.

\begin{defn}
  Let $D_\kappa$ denote the naturally ordered set of all symbols
  $1/\alpha$ for $0<\alpha<\kappa$, together with $0$.  A space $X$
  together with a map $d\colon X\times X\to D_\kappa$ is said to be a
  \emph{$\kappa$-ultrametric space} iff the usual axioms hold:
  \begin{itemize}
  \item $d(x,y)=0$ iff $x=y$;
  \item $d(x,y)=d(y,x)$; and
  \item $d(x,z)\leq\max(d(x,y),d(y,z))$.
  \end{itemize}
  We say that a sequence $(x_\alpha)_{\alpha<\lambda}$ is
  \emph{$\kappa$-Cauchy} if $d(x_\alpha,x_\beta)\to0$ whenever
  $\alpha,\beta\to\lambda$ (here $\lambda\leq\kappa$), and we say that
  $X$ is \emph{complete} if every $\kappa$-Cauchy sequence converges.
\end{defn}

For example, ${}^\kappa\kappa$ is a complete $\kappa$-ultrametric space with the metric $d(x,y)=1/\alpha$, where $\alpha$ is the length of the largest common initial segment of $x$ and $y$.  Also, the subset $\Sym(\kappa)\subset\kappa^\kappa$ consisting of all bijections of $\kappa$ is a complete $\kappa$-ultrametric space with the metric given by $\max(d(x,y),d(x^{-1},y^{-1}))$.

It has been shown in \cite[Section~3]{Sikorski} that ${}^\kappa\kappa$ contains a homeomorphic copy of every $\kappa$-ultrametric space of weight $\leq\kappa$.  In fact, we will show that ${}^\kappa\kappa$ exhibits a high degree of homogeneity and plays the role of a universal Urysohn space in this context.

Before stating this result, we take a moment to generalize the notion of $\kappa$-ultrametric spaces a little further.  Let $R$ be any linearly ordered set with a minimal zero element.  Then an \emph{$R$-ultrametric space} is defined analogously with $\kappa$-ultrametric spaces, but with the set of metric values $D_\kappa$ replaced by the set $R$.  Ultrametric spaces have been studied in even greater generality; for instance in \cite{priess} the authors consider metrics which take values in a partially ordered set.  We confine ourselves to the linearly ordered case.  We remark that these generalized ultrametric spaces share many of the basic properties of ordinary ultrametric spaces, for instance:

\begin{itemize} 
\item For all $x,y,z$, two of the distances $d(x,y)$, $d(x,z)$, and $d(y,z)$ are equal, and the third is no greater;
\item Every open ball is closed;
\item Any two open balls are either disjoint or nested, and;
\item The set of metric values on any dense subset of the space is
  equal to the set of metric values on the whole space
\end{itemize} 

In the remainder of the section, we will show that for each reasonable set of metric values $R$ there exists a universal Urysohn space among the $R$-ultrametric spaces.

\begin{defn}
  An $R$-ultrametric space $U$ is said to be \emph{Urysohn} iff $U$ is complete, has density $\leq\kappa$, and satisfies the \emph{extension property}: for every $X\subset U$ of cardinality $<\kappa$ and for every one-point ($R$-ultrametric) extension $X\cup\{a\}$, there exists $u\in U$ such that $d(x,a)=d(x,u)$ for all $x\in X$.
\end{defn}

The completeness and extension properties of the Urysohn space $U_R$ of course guarantee that it is universal for $R$-ultrametric spaces of density $\leq\kappa$. As we have mentioned above, the $\kappa$-ultrametric Urysohn space is simply ${}^\kappa\kappa$ itself.

\begin{prop}
  ${}^\kappa\kappa$ is a $\kappa$-ultrametric Urysohn space.
\end{prop}

\begin{proof}
  We verify that ${}^\kappa\kappa$ has the extension property. Let $\lambda<\kappa$ and suppose that $x_\alpha\in{}^\kappa\kappa$ for $\alpha<\lambda$. Let $X\cup\{a\}$ be a $\kappa$-ultrametric extension of $X$. We inductively build a branch $u\in{}^\kappa\kappa$ such that $d(x,a)=d(x,u)$ for all $x\in X$. At each stage $\alpha$ we must select $u(\alpha)$ to agree with all those $x\in X$ such that $d(x,a)<1/\alpha$ and disagree with all $x\in X$ such that $d(x,a)=1/\alpha$.

  If this prescription fails, let $\alpha$ be the least place where it gives contradictory instructions. Then either $u$ is required to agree with $x(\alpha)$ and disagree with $x'(\alpha)$ but $x(\alpha)=x'(\alpha)$, or $u(\alpha)$ is required to agree with both $x(\alpha)$ and $x'(\alpha)$ but $x(\alpha)\neq x'(\alpha)$. In the first case we have $d(a,x)<1/\alpha$ and $d(a,x')=1/\alpha$, so $d(x,x')\leq1/\alpha$. But then $x(\alpha)=x'(\alpha)$ implies that $d(x,x')<1/\alpha$, which in turn implies that $d(a,x')<1/\alpha$ after all, a contradiction. In the second case we have $d(a,x),d(a,x')<1/\alpha$ and so $d(x,x')<1/\alpha$, contradicting that $x(\alpha)\neq x'(\alpha)$.
\end{proof}

The following result states that a Urysohn space exists for any reasonable set of distances $R$. As always, we are assuming that $\kappa^{<\kappa}=\kappa$.

\begin{thm}
  \label{thm:urysohn}
  If $R$ is a set of distances which admits greatest lower bounds and such that $R\smallsetminus0$ has coinitiality $\kappa$, then there exists an $R$-ultrametric Urysohn space $U_R$.
\end{thm}

Our proof of this result follows the Kat\v{e}tov-style construction of an ordinary ultrametric Urysohn space given in \cite{shao}.  This version of the proof will allow us to conclude that the isometry group of each Urysohn space is universal among isometry groups.  Since the details of the construction are very similar to the argument found there, we give only a short outline.

To begin, if $X$ is a $\kappa$-ultrametric space, we say that a function $f\colon X\to R$ is \emph{Kat\v{e}tov} if for all $x,y\in X$:
\begin{align*}
  d(x,y)&\leq\max(f(x),f(y))\;,\\
  f(x)&\leq\max(d(x,y),f(y)),\text{ and}\\
  f(y)&\leq\max(d(x,y),f(x))\;.
\end{align*}
The Kat\v{e}tov functions thus correspond to the one-point extensions
of $X$.  If $Z\subset X$, we say that a Kat\v{e}tov function is
\emph{supported} on $Z\subset X$ if for all $x\in X$:
\begin{equation*}
  f(x)=
  \begin{cases}
    \max(f(z),d(z,x))&\text{if }z\in Z\text{ and }f(z)\neq d(z,x)\\
    \inf_{z\in Z}f(z)&\text{if for all }z\in Z,f(z)=d(z,x)
  \end{cases}
\end{equation*}
The following lemma states that Kat\v{e}tov functions are plentiful.

\begin{lem}[cf.\ {\protect \cite[Theorem 5.5]{shao}}]
  \label{lem:supports}
  If $Z\subset X$, then any Kat\v{e}tov partial function on $Z$
  extends to a Kat\v{e}tov function on $X$ which is supported on $Z$.
\end{lem}

If $X$ is an $R$-ultrametric space then the set $E(X)$ of Kat\v{e}tov functions on $X$ with a support of size $<\kappa$ is naturally an $R$-ultrametric space (with $d(f,g)=\max(f(x),g(x))$ where $x$ is such that $f(x)\neq g(x)$).  Moreover, identifying each $x\in X$ with the function $f_x(z)=d(x,z)$ gives a natural embedding of $X$ as a subspace of $E(X)$.  The next lemma states that assuming $\kappa^{<\kappa}=\kappa$, if $X$ has density $\leq\kappa$, then so does $E(X)$.

\begin{lem}[cf.\ \protect{\cite[Theorem~5.9]{shao}}]
  Assume as usual that $\kappa^{<\kappa}=\kappa$.  If $D$ is a dense subset of $X$ of size $\leq\kappa$, let $E(X,D)$ denote the subspace of $E(X)$ consisting of Kat\v{e}tov functions with a support contained in $D$.  Then $E(X)\smallsetminus X\subset E(X,D)$ and hence has size $\leq\kappa$.
\end{lem}

Given any $R$-ultrametric space $X_0$ of density $\leq\kappa$, we now let $X_{\alpha+1}=E(X_\alpha)$, and $X_\lambda=\bigcup_{\alpha<\lambda}X_\alpha$ for $\lambda$ limit.  It follows easily from the above results that $X_\kappa$ satisfies the extension property.  While it may not be complete, it follows from \cite[Proposition 2.4]{shao} that its completion will then be an $R$-ultrametric Urysohn space.  This concludes the outline of the proof of Theorem~\ref{thm:urysohn}.

As in the classical case, an argument due to Uspenskij shows that the group of isometries of the Urysohn space is universal among isometry groups.

\begin{cor}
  If $X$ is an $R$-ultrametric space with density $\leq\kappa$ then there is a continuous embedding from $\Iso(X)$ into $\Iso(U_R)$.
\end{cor}



We remark that the $R$-ultrametric Urysohn spaces even satisfy the extension property for $\kappa$-compact subsets. To see this, cover a given compact set with a family of small open balls and choose a subcover of size $<\kappa$. Then apply the extension property to the set of centers of these balls.


We conclude with the question of whether there exists a Urysohn-type space for a broader class of spaces than just the $\kappa$-ultrametric spaces. For example, the space ${}^\kappa\kappa$ with the lexicographic topology is an important space that is not even $\kappa$-additive.

\begin{question}
  Is there a homogeneous strong $\kappa$-Choquet space of weight $\leq\kappa$ which contains a copy of $({}^{\kappa}\kappa,\mathsf{lex})$?
\end{question}

Here ``homogeneous'' can mean having the extension property, or for a simpler question it can mean that for any $x,y$ there is a homeomorphism mapping $x$ to $y$.

\section{Hereditarily Baire spaces}
\label{section hereditarily baire}

In this section we, we characterize the subspaces $X$ of ${}^{\kappa}\kappa$ such that every spherically closed (see below) subset of $X$ is $\kappa$-Baire.  This extends a result of Debs for separable spaces, see \cite[Theoreme~3.2]{debs}.

\begin{defn} 
  \begin{enumerate} 
  \item A $\kappa$-ultrametric space is \emph{spherically ($\mu$-, $\mathord{<}\mu$-, $\mathord{\leq}\mu$-) complete} if the intersection of every decreasing sequence (of length $\mu$, $\mathord{<}\mu$, $\mathord{\leq}\mu$) of open balls is nonempty.
  \item A subset $A$ of a $\kappa$-ultrametric space $X$ is \emph{spherically ($\mu$-, $\mathord{<}\mu$-, $\mathord{\leq}\mu$-) closed} in $X$ if whenever $B_\alpha$ is a decreasing sequence of open balls (of length $\mu$, $\mathord{<}\mu$, $\mathord{\leq}\mu$) such that each $B_\alpha$ meets $A$ and $\bigcap B_\alpha\neq\emptyset$ then $\bigcap B_\alpha\cap A\neq\emptyset$.
  \item A $\kappa$-ultrametric space $X$ is \emph{hereditarily (weakly) $\kappa$-Baire} if every spherically closed subset of $X$ is (weakly) $\kappa$-Baire.
  \end{enumerate} 
\end{defn} 

We begin with several remarks on these definitions.  It is clear that every $\mathord{\leq}\kappa$-spherically complete space is strong $\kappa$-Choquet.  In fact, $X$ is spherically $\mathord{\leq}\kappa$-complete if and only if player~II wins every run of the version of the weak $\kappa$-Choquet game in which both players can only play open balls.

There are surprising examples of spherically closed subsets of
${}^\kappa\kappa$.  For example, by \cite{luecke-schlicht-kurepa-trees} it is
consistent that there is a $\mathord{<}\kappa$-closed $\kappa$-Kurepa
tree $T\subset {}^{<\kappa}\kappa$ for arbitrarily large $\kappa$.
It follows that $[T]$ is spherically closed for such $T$; in fact, a
subset $A\subset{}^{\kappa}\kappa$ is spherically
$\mathord{<}\kappa$-closed in ${}^\kappa\kappa$ iff the tree
$T=\set{s\in {}^{<\kappa}\kappa\mid (\exists x\in A)\ s\subset x}$
is $\mathord{<}\kappa$-closed.

It also follows from this last fact that ${}^\kappa\kappa$ is an example of a hereditarily $\kappa$-Baire space.  But there do exist $\kappa$-Baire spaces $X$ which are not hereditarily $\kappa$-Baire.  For example, let $Q\subset{}^{\kappa}\kappa$ be a dense subset of size $\kappa$.  We shall define a space $X$ as a union of copies $X_q$ of ${}^\kappa\kappa$ for $q\in Q$.  First, realize $Q$ as a subspace of $X$ by identifying $q\in Q$ with $0^\kappa\in X_q$.  Then, extend the usual ultrametrics on $Q$ and $X_q$ to all of $X$ by $d(x,x')=\max(d(x,q),d(q,q'),d(q',x'))$ for $x\in X_q$ and $x'\in X_{q'}$ and $q\neq q'$.  Now, $X$ is a $\kappa$-Baire $\kappa$-ultrametric space of density $\leq\kappa$.  And $Q$ is a $\kappa$-perfect subset of $X$ which is is spherically closed in $X$, but $Q$ is not $\kappa$-Baire.

The following result shows that the example in the previous paragraph is in some sense typical of spaces which are $\kappa$-Baire but not hereditarily $\kappa$-Baire.

\begin{thm}
  \label{characterization of hereditarily Baire}
  If $X\subset {}^{\kappa}\kappa$, then the following are equivalent:
  \begin{enumerate} 
  \item $X$ is not hereditarily $\kappa$-Baire.
  \item There is a subset $Q\subset X$ such that $Q$ is spherically closed in $X$, $Q$ has size $\kappa$, and $Q$ has no $\kappa$-isolated points.
  \end{enumerate} 
\end{thm}

We remark that $X\subset{}^\kappa\kappa$ is hereditarily $\kappa$-Baire if and only if it is hereditarily weakly $\kappa$-Baire, since basic open subsets of ${}^\kappa\kappa$ are spherically closed.  It is also worth noting that the result could also be stated for $X$ a $\kappa$-ultrametric space of density $\leq\kappa$, since any such space is homeomorphic to a subspace of ${}^{\kappa}\kappa$.

\begin{proof}[Proof of Theorem~\ref{characterization of hereditarily Baire}]
  (b)$\to$(a) Clearly if there is such a subspace $Q$, then $Q$ is not $\kappa$-Baire and hence $X$ is not hereditarily $\kappa$-Baire.

  (a)$\to$(b) Suppose that $X$ is not weakly hereditarily $\kappa$-Baire.  Let $Y\subset X$ be a spherically closed subset of $X$ which is not weakly $\kappa$-Baire. Then by Theorem~\ref{characterization of Baire by games}, player~I has a winning strategy $\sigma$ in the modified version of the basic $\kappa$-Choquet game for $Y$ where player~II begins and plays at limits but the winning condition remains the same. Recall that $N_s$ denotes the basic open subset ${}^{\kappa}\kappa$ with stem $s$; we will write $U_s= N_s\cap Y$ in this proof. 


  We construct a tree $T\subset {}^{<\kappa}\kappa$ and $(r_t,s_t,x_t)_{t\in T}$ such that $r_t,s_t\in {}^{<\kappa}\kappa$, $x_t\in Y$ and for all $t\in T$:
\begin{itemize}
\item $U_{s_t}\neq\emptyset$;
\item $x_t\in U_{s_t}$ and $x_t$ is nonisolated in $Y$;
\item for every branch $b$ in $T$ of length $\gamma$, the sequence $(U_{r_{b\upharpoonright\alpha}},U_{s_{b\upharpoonright\alpha}})_{\alpha<\gamma}$ is a valid run according to $\sigma$; and
\item if $t\in T$, then $t^\smallfrown\alpha\in T$ for all $\alpha<\kappa$.
\end{itemize} 
To begin, let $r_{\emptyset}=\emptyset$ and choose $s_{\emptyset}$ so that $U_{s_{\emptyset}}$ is the answer of $\sigma$ to $U_{r_{\emptyset}}$. Since $\sigma$ is a winning strategy for player~I, $U_{s_{\emptyset}}$ cannot have any isolated (equivalently $\kappa$-isolated) points, since otherwise player~II could win by playing this singleton point.  In particular, we may let $x_{\emptyset}\in U_{s_{\emptyset}}$ be nonisolated in $Y$.

Next suppose that $t\in T$ and that $r_{t'},s_{t'},x_{t'}$ are defined for all $t'\subset t$.  Since $x_t\in U_{s_t}$ is nonisolated in $Y$, can find a sequence $(r_{t^\smallfrown\alpha})_{\alpha<\kappa}$ such that $s_t\subset r_{t^\smallfrown\alpha}$, $r_{t^\smallfrown\alpha}\not\subset x_t$, the intersections $(r_{t^\smallfrown\alpha}\cap x_t)_{\alpha<\kappa}$ form a strictly increasing sequence in ${}^{<\kappa}\kappa$, and $U_{r_{t^\smallfrown\alpha}}$ is nonempty for all $\alpha<\kappa$.  We then choose $s_{t^\smallfrown\alpha}$ such that $U_{s_{t^\smallfrown\alpha}}$ is the answer of $\sigma$ to the run beginning with the sets $U_{r_{t'}}, U_{s_{t'}}$ for $t'\subset t$ and ending with $U_{r_{t^\smallfrown\alpha}}$.  Again, since $\sigma$ is a winning strategy for player~I, we may choose $x_{t^\smallfrown\alpha}\in U_{s_{t^\smallfrown\alpha}}$ nonisolated in $Y$ for all $\alpha<\kappa$.

  Finally suppose that $t\in {}^{<\kappa}\kappa$ has limit length, that $t'\in T$ for all $t'\subsetneq t$, and that $r_{t'},s_{t'}$ are defined for all $t'\subsetneq s$.  Let $s=\bigcup_{t'\subsetneq t} s_{t'}$. If $U_s$ is empty, let $t\notin T$.  If $U_s$ is nonempty, let $t\in T$ and $r_t=s$.  Then choose $s_t$ such that $U_{s_t}$ is the answer of $\sigma$ to the run beginning with the sets $U_{r_{t'}}, U_{s_{t'}}$ for $t'\subset t$ in $S$ and ending with $U_{r_t}$.  Again, there is some $x_t\in U_{s_t}$ nonisolated in $Y$, since $\sigma$ is a winning strategy for player~I. This completes the construction.

  Now, we let $Q=\set{x_t\mid t\in T}$.  The successor stages of the construction ensure that $Q$ has no isolated points, and equivalently no $\kappa$-isolated points, and hence that $Q$ is $\kappa$-perfect.  Thus it remains only to show that $Q$ is spherically closed in $Y$, and hence in $X$.  We begin by showing that $Q$ is spherically $\mathord{<}\kappa$-closed in $Y$.  In other words, letting $S=\set{x_t\upharpoonright\alpha:t\in T,\,\alpha<\kappa}$, $(s_{\alpha})_{\alpha<\delta}$ a strictly increasing sequence in $S$ (with $\delta<\kappa$ regular), and $u=\bigcup_{\alpha<\delta} s_{\alpha}$, we must show that if $U_u\neq\emptyset$ then $U_u$ contains an element of $Q$.
  
  To begin, we know that for each $\alpha<\delta$, since $s_{\alpha}\in S$ there is some $x\in Q$ with $s_{\alpha}\subset x$. Let $t_{\alpha}\in T$ be minimal such that $s_{\alpha}\subset x_{t_{\alpha}}$. Passing to a subsequence if necessary, we can assume that $s_{\beta}\not\subset x_{t_{\alpha}}$ whenever $\alpha<\beta<\delta$.
  
  \begin{claim}\label{claim: strictly increasing sequence in T}
    $t_{\alpha}\subsetneq t_{\beta}$ whenever $\alpha<\beta<\delta$.
  \end{claim} 
  
  \begin{claimproof}
  Let $\gamma=\sup\set{\delta<\kappa\mid t_{\alpha}\upharpoonright \delta=t_{\beta}\upharpoonright\delta}$ and $t=t_{\alpha}\upharpoonright\gamma=t_{\beta}\upharpoonright\gamma$. Let $\zeta=\sup\set{\delta<\kappa\mid x_{t_{\alpha}}\upharpoonright \delta= x_{t_{\beta}}\upharpoonright\delta}$. Then $x_t\supset x_{t_{\alpha}}\upharpoonright \zeta= x_{t_{\beta}}\upharpoonright \zeta$ by the construction. Since $s_{\alpha}\subset x_{t_{\alpha}}$ and $s_{\alpha}\subset x_{t_{\beta}}$, this implies that $s_{\alpha}\subset x_t$. Since $t_{\alpha}$ is minimal with $s_{\alpha}\subset x_{t_{\alpha}}$, this implies that $t_{\alpha}=t$.  Thus $t_{\alpha}\subset t_{\beta}$, and since $x_{t_{\alpha}}\neq x_{t_{\beta}}$ we must have $t_{\alpha}\subsetneq t_{\beta}$.
  \end{claimproof}
  
  \begin{claim}\label{claim: branch follows the strategy} $u=\bigcup_{\alpha<\delta} s_{\alpha}=\bigcup_{\alpha<\delta} r_{t_{\alpha}}=\bigcup_{\alpha<\delta} s_{t_{\alpha}}$. 
  \end{claim} 
  
  \begin{claimproof}
    We have $s_{\alpha}\subset r_{t_{\beta}}\subset s_{t_{\beta}}$ and $r_{t_{\alpha}}\subset s_{t_{\alpha}}\subset s_{\beta}$ for all $\alpha<\beta<\delta$ by the construction of $T$.   
  \end{claimproof}
  
  Now, let $v=\bigcup_{\alpha<\delta} t_{\alpha}$. Then by the construction, if $U_u\neq\emptyset$ we will have that $v\in T$ and $x_v\in U_u$. This concludes the proof that $Q$ is spherically $\mathord{<}\kappa$-closed in $Y$.
    
  Finally, we show that $Q$ is spherically $\kappa$-closed in $Y$.  Suppose that $(s_{\alpha})_{\alpha<\kappa}$ is strictly increasing in $S$, and as before let $t_{\alpha}\in T$ be minimal such that $s_{\alpha}\subset x_{t_{\alpha}}$. Again assume that $s_{\beta}\not\subset x_{t_{\alpha}}$ for all $\alpha<\beta<\kappa$. Then by the proof of Claim~\ref{claim: strictly increasing sequence in T}, $t_{\alpha}\subsetneq t_{\beta}$ if $\alpha<\beta<\kappa$.  And by the proof of Claim~\ref{claim: branch follows the strategy}, we have $s_{\alpha}\subset r_{t_{\beta}}\subset s_{t_{\beta}}$ and $r_{t_{\alpha}}\subset s_{t_{\alpha}}\subset s_{\beta}$ for all $\alpha<\beta<\kappa$. Since $(r_{t_{\alpha}})_{\alpha<\kappa}$ is in accordance with a run of the winning strategy $\sigma$ for player~I, we have $\bigcup_{\alpha<\kappa}s_{\alpha}=\bigcup_{\alpha<\kappa} r_{t_{\alpha}}\notin Y$. This concludes the proof that $Q$ is spherically closed in $Y$.
\end{proof} 

The next result shows that for $\kappa>\omega$ there are many different possible spaces $Q\subset {}^{\kappa}2$ which can arise in part (b) of Theorem~\ref{characterization of hereditarily Baire}. In particular, any space $Q$ as constructed in the next proposition has size $\kappa$, has no $\kappa$-isolated points, and is (trivially) spherically closed in $X=Q$. 
This contrasts with the case $\kappa=\omega$, when of course any such $Q$ is homeomorphic to $\mathbb{Q}$.

\begin{prop}
 There is a sequence $(X_{\alpha})_{\alpha<2^{\kappa}}$ of subsets of ${}^{\kappa}2$ of size $\kappa$ without $\kappa$-isolated points such that $X_{\alpha}$ is not homeomorphic to a spherically closed subset of $X_{\beta}$ for all $\alpha,\beta<2^{\kappa}$ with $\alpha\neq\beta$. 
\end{prop} 

\begin{proof}
  Suppose that $A\subset\kappa$. Let $X_A$ denote the set of $x\in {}^{\kappa}2$ such that $x$ is the characteristic function of a set $s_x\subset \kappa$ and $\mathrm{lim}(s_x)\cap A=\emptyset$, where $\mathrm{lim}(s_x)$ denotes the set of limit points of elements of $s_x$.  If $A$ is stationary in $\kappa$, then $s_x$ is bounded for all $x\in X_A$, so $|X_A|=\kappa$ by our assumption that $\kappa^{<\kappa}=\kappa$.

Suppose that $A$, $B$ are disjoint stationary subsets of $\kappa$. 
We claim that for all $s_0\in {}^{<\kappa}\kappa$ with $X_A\cap N_{s_0}\neq\emptyset$, there is no function $f\colon X_{A}\cap N_{s_0}\rightarrow X_{B}$ which embeds $X_{A}\cap N_{s_0}$ homeomorphically as a spherically closed subset of $X_{B}$. Indeed, if there were such a function $f$, we will construct strictly increasing continuous sequences $(s_{\alpha})_{\alpha<\kappa}$ and $(t_{\alpha})_{\alpha<\kappa}$ in ${}^{<\kappa}\kappa$ such that for all $\alpha<\kappa$:
\begin{enumerate} 
\item $f[X_A\cap N_{s_{\alpha+1}}]\subset N_{t_{\alpha}}$, 
\item $X_{A}\cap N_{s_{\alpha}} \neq\emptyset$, 
\item $\mathrm{dom}(s_{\alpha+1})> \mathrm{dom}(t_{\alpha})$, 
\item $s_{\alpha+1}(\beta)=1$ for some $\beta\in
  \dom(s_{\alpha+1})\smallsetminus\dom(s_{\alpha})$;
\end{enumerate} 
and symmetrically:
\begin{enumerate} 
\setcounter{enumi}{4}
\item $f^{-1}[X_A\cap N_{t_{\alpha+1}}]\subset N_{s_{\alpha+1}}$, 
\item $f[X_A\cap N_{s_0}]\cap N_{t_{\alpha}} \neq\emptyset$, 
\item $\mathrm{dom}(t_{\alpha+1})> \mathrm{dom}(s_{\alpha+1})$, and 
\item $t_{\alpha+1}(\beta)=1$ for some $\beta\in\mathrm{dom}(t_{\alpha+1})\smallsetminus\mathrm{dom}(t_{\alpha})$.
\end{enumerate} 
We choose an arbitrary $t_0\in {}^{<\kappa}\kappa$ with $f[N_{s_0}]\cap N_{t_0}\neq 0$. 
The construction can be carried out in successor steps since $f$ and $f^{-1}$ are continuous. 
Suppose that $\gamma<\kappa$ is a limit and that we have constructed $(s_{\alpha}, t_{\alpha})_{\alpha<\gamma}$. Let $s_{\gamma}=\bigcup_{\alpha<\gamma} s_{\alpha}$ and $t_{\gamma}=\bigcup_{\alpha<\gamma} t_{\alpha}$. Then $\mathrm{dom}(s_{\gamma})=\mathrm{dom}(t_{\gamma})$ and $f[X_A\cap N_{s_{\gamma}}]=f[X_A\cap N_{s_0}]\cap N_{t_{\gamma}}$. 
We claim that one of the sets $X_A\cap N_{s_{\gamma}}$ and $f[X_A\cap N_{s_0}]\cap N_{t_{\gamma}}$, and therefore both, are nonempty. 
First suppose that $\gamma\notin A$. Since $X_A\cap N_{s_{\alpha}} \neq\emptyset$ for all $\alpha<\gamma$, we have $X_A\cap N_{s_{\gamma}}\neq\emptyset$. If on the other hand $\gamma\in A$, then $\gamma\notin B$, since $A,B$ are disjoint. Since $f[X_A\cap N_{s_0}]\cap N_{t_{\alpha}}\neq\emptyset$ for all $\alpha<\gamma$, and since $f[X_A\cap N_{s_0}]$ is spherically closed in $X_B$, we have $f[X_A\cap N_{s_0}]\cap N_{t_{\gamma}}\neq\emptyset$. 
Therefore the construction can be continued for $\kappa$ steps.

We have thus constructed $x=\bigcup_{\alpha<\kappa} s_{\alpha} \in X_{A}$ and $y=\bigcup_{\alpha<\kappa} t_{\alpha}\in X_{B}$ with $f(x)=y$. Suppose that $x$ is the characteristic function of $s_x\subset \kappa$. Then $\lim(s_x)$ is a club and hence $\lim(s_x)\cap A\neq\emptyset$, contradicting the fact that $x\in X_A$.

Finally, suppose that $(C_{\alpha})_{\alpha<2^{\kappa}}$ is a sequence of subsets of $\kappa$ with $C_{\alpha}\not\subset C_{\beta}$ for all $\alpha,\beta<2^{\kappa}$ with $\alpha\neq \beta$.  Suppose that $(A_{\alpha})_{\alpha<\kappa}$ is a sequence of pairwise disjoint stationary subsets of $\kappa$.  Let $X^C=\bigsqcup_{\alpha\in C} X_{A_{\alpha}}$ for $C\subset\kappa$.

We claim that $X^{C_{\alpha}}$ cannot be embedded as a spherically closed subset of $X^{C_{\beta}}$ for all $\alpha,\beta<2^{\kappa}$ with $\alpha\neq\beta$.  Suppose that $f\colon X^{C_{\alpha}}\rightarrow X^{C_{\beta}}$ is such an embedding.  Suppose that $\gamma\in C_{\alpha}\smallsetminus C_{\beta}$ and $\delta\in C_{\beta}$ with $f[X_{A_{\gamma}}]\cap X_{A_{\delta}}\neq\emptyset$.  Since $f$ is a homeomorphism, we can find some $s,t\in {}^{<\kappa}2$ such that $N_s\cap X_{A_{\gamma}}\neq\emptyset$ and $f[N_s\cap X_{A_{\gamma}}]=N_t\cap X_{A_{\delta}}$.  However, we have seen that $X_{A_{\gamma}}$ does not embed as a spherically closed subset of $X_{A_{\delta}}$.
\end{proof}

\bibliographystyle{alpha}
\begin{singlespace}
  \bibliography{choquet}

\begin{thebibliography}{FHK11}

\bibitem[Cho69]{choquet}
Gustave Choquet.
\newblock {\em Lectures on analysis. {V}ol. {I}: {I}ntegration and topological
  vector spaces}.
\newblock Edited by J. Marsden, T. Lance and S. Gelbart. W. A. Benjamin, Inc.,
  New York-Amsterdam, 1969.

\bibitem[Deb88]{debs}
G.~Debs.
\newblock Espaces h\'er\'editairement de {B}aire.
\newblock {\em Fund. Math.}, 129(3):199--206, 1988.

\bibitem[Del84]{delon}
Fran{\c{c}}oise Delon.
\newblock Espaces ultram\'etriques.
\newblock {\em J. Symbolic Logic}, 49(2):405--424, 1984.

\bibitem[FHK11]{vadim}
Sy~David Friedman, Tapani Hyttinen, and Vadim Kulikov.
\newblock Borel equivalence relations and classifications of countable models.
\newblock {\em Preprint}, 2011.

\bibitem[Gao09]{gao}
Su~Gao.
\newblock {\em Invariant descriptive set theory}, volume 293 of {\em Pure and
  Applied Mathematics (Boca Raton)}.
\newblock CRC Press, Boca Raton, FL, 2009.

\bibitem[GJ76]{gillman}
Leonard Gillman and Meyer Jerison.
\newblock {\em Rings of continuous functions}.
\newblock Springer-Verlag, New York, 1976.
\newblock Reprint of the 1960 edition, Graduate Texts in Mathematics, No. 43.

\bibitem[GS11]{shao}
Su~Gao and Chuang Shao.
\newblock Polish ultrametric {U}rysohn spaces and their isometry groups.
\newblock {\em Topology Appl.}, 158(3):492--508, 2011.

\bibitem[HS01]{halko}
Aapo Halko and Saharon Shelah.
\newblock On strong measure zero subsets of {$^\kappa2$}.
\newblock {\em Fund. Math.}, 170(3):219--229, 2001.

\bibitem[Ilm11]{Ilmavirta}
Joonos Ilmavirta.
\newblock Urysohn's metrization theorem for higher cardinals.
\newblock {\em Submitted}, 2011.

\bibitem[Kec95]{kechris}
Alexander~S. Kechris.
\newblock {\em Classical descriptive set theory}, volume 156 of {\em Graduate
  Texts in Mathematics}.
\newblock Springer-Verlag, New York, 1995.

\bibitem[LS13a]{luecke-schlicht-continuous-images}
Philipp L\"ucke and Philipp Schlicht.
\newblock Continuous images of closed sets in generalized {B}aire spaces.
\newblock {\em in preparation}, 2013.

\bibitem[LS13b]{luecke-schlicht-kurepa-trees}
Philipp L\"ucke and Philipp Schlicht.
\newblock Kurepa trees as continuous images of generalized {B}aire spaces.
\newblock {\em in preparation}, 2013.

\bibitem[MR13]{mottoros}
Luca Motto~Ros.
\newblock The descriptive set-theoretical complexity of the embeddability
  relation on models of large size.
\newblock {\em Ann. Pure Appl. Logic}, 164(12):1454--1492, 2013.

\bibitem[MV93]{vaan}
Alan Mekler and Jouko V{\"a}{\"a}n{\"a}nen.
\newblock Trees and {$\Pi\sp 1\sb 1$}-subsets of {$\sp {\omega\sb 1}\omega\sb
  1$}.
\newblock {\em J. Symbolic Logic}, 58(3):1052--1070, 1993.

\bibitem[PCR95]{priess}
Sibylla Priess-Crampe and Paulo Ribenboim.
\newblock Equivalence relations and spherically complete ultrametric spaces.
\newblock {\em C. R. Acad. Sci. Paris S\'er. I Math.}, 320(10):1187--1192,
  1995.

\bibitem[Sch13]{schlicht-perfect}
Philipp Schlicht.
\newblock Perfect subsets of generalized {B}aire spaces.
\newblock {\em In preparation}, 2013.

\bibitem[Sik50]{Sikorski}
Roman Sikorski.
\newblock Remarks on some topological spaces of high power.
\newblock {\em Fund. Math.}, 37:125--136, 1950.

\bibitem[Tuu92]{tuuri}
Heikki Tuuri.
\newblock Relative separation theorems for {$\mathcal{L}_{\kappa^+\kappa}$}.
\newblock {\em Notre Dame J. Formal Logic}, 33(3):383--401, 1992.

\bibitem[Vau75]{vaught}
Robert Vaught.
\newblock Invariant sets in topology and logic.
\newblock {\em Fund. Math.}, 82:269--294, 1974/75.
\newblock Collection of articles dedicated to Andrzej Mostowski on his sixtieth
  birthday, VII.

\end{thebibliography}
\end{singlespace}

\end{document}